\documentclass[12pt,oneside,english]{amsart}
\usepackage{fljegyzet}
\usepackage{accents}

\usepackage{url}
\usepackage[breaklinks]{hyperref}
\DeclareMathOperator\td{Td}
\DeclareMathOperator\Td{Td}
\DeclareMathOperator\coll{co}
\DeclareMathOperator\mC{mC}
\DeclareMathOperator\mc{mC}

\DeclareMathOperator\csm{c^{sm}}

\DeclareMathOperator\T{\mathbb T}

\def\omeg{\includegraphics[width=5mm]{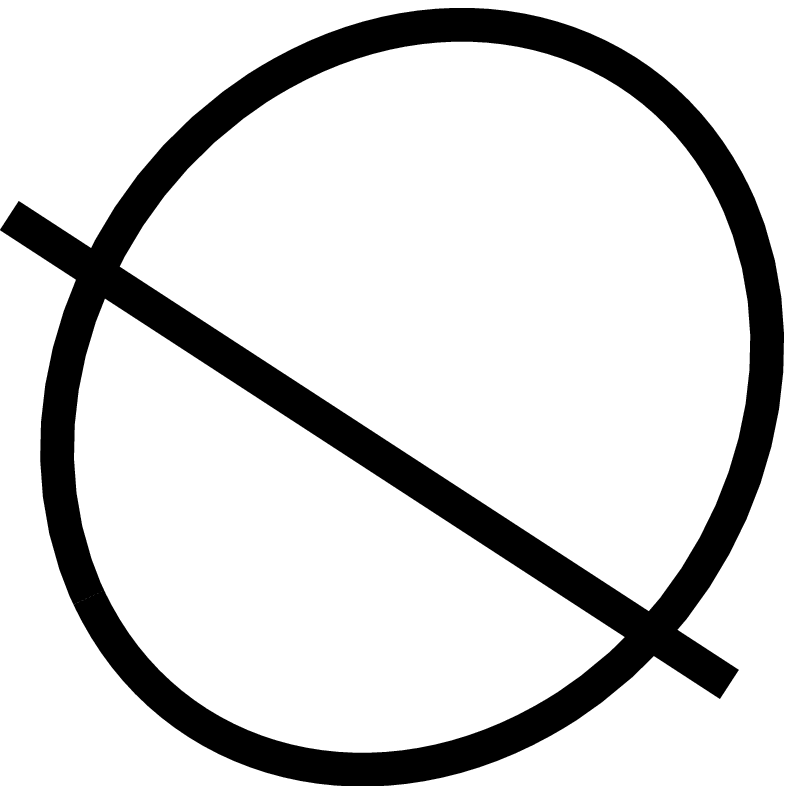}}
\def\tri{\includegraphics[width=5mm]{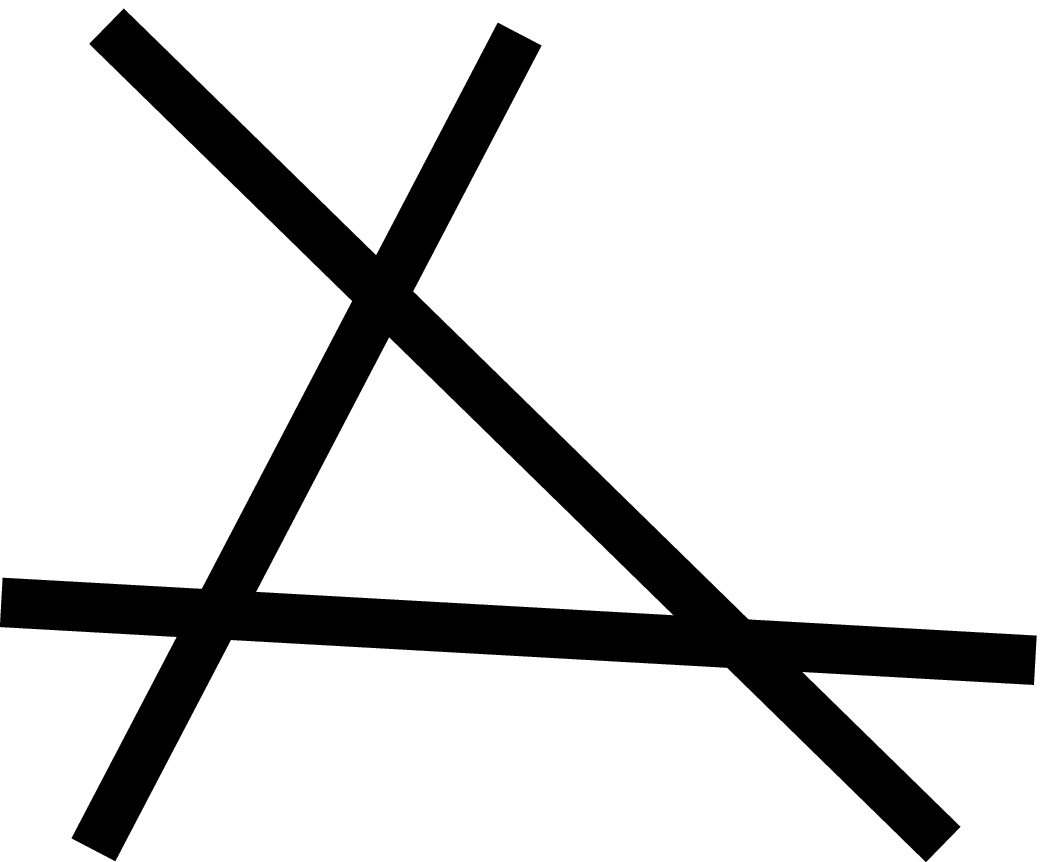}}
\def\tangent{\includegraphics[width=5mm]{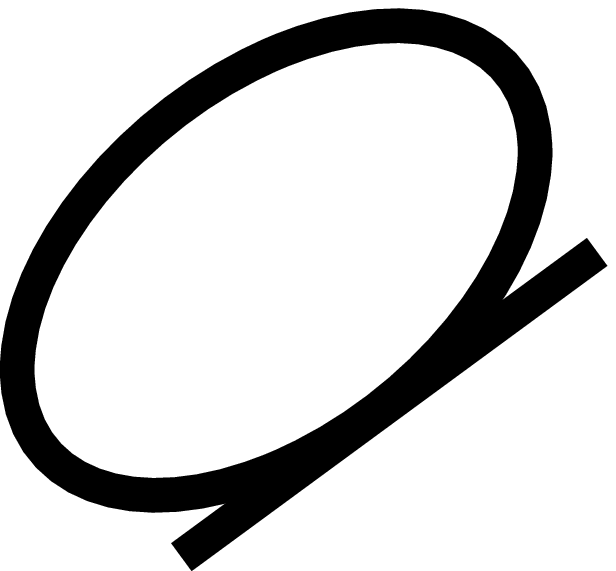}}
\def\concurrent{\includegraphics[width=5mm]{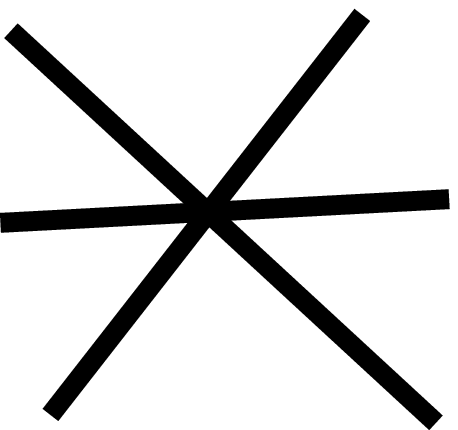}}
\def\cusp{\includegraphics[width=5mm]{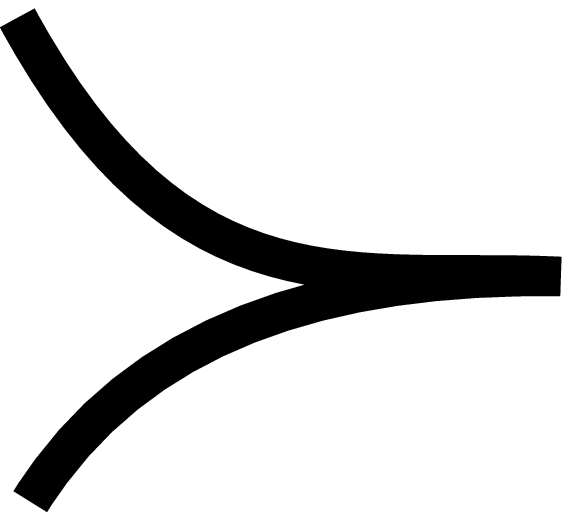}}
\def\node{\includegraphics[width=6mm]{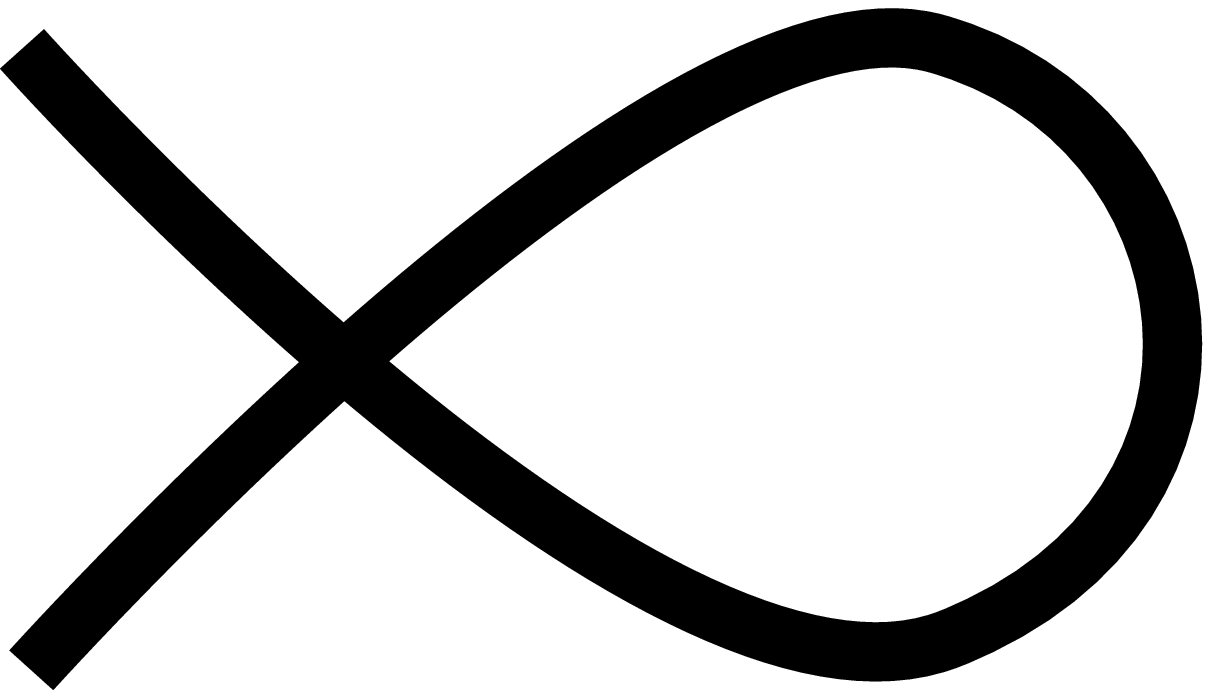}}

\author{L\'aszl\'o M. Feh\'er}
\address{Institute of Mathematics, E\"otv\"os University Budapest, Hungary}
\email{lfeher63@gmail.com}


\begin{document}
\begin{abstract}
  We study motivic Chern classes of cones. First we show examples of projective cones of smooth curves such that their various $K$-classes (sheaf theoretic, push-forward and motivic) are all different. Then we show connections between the torus equivariant motivic Chern class of a projective variety and of its affine cone, generalizing results on projective Thom polynomials.
\end{abstract}
\title{Motivic Chern classes of cones}
\maketitle
\tableofcontents

\section{Introduction}
We study two different topics in this paper. The common technical issue is to understand the motivic properties of cones. Equivariant motivic classes of cones were studied in \cite{Weber=EquivariantHirzebruch}   previously. Our results are related, but the philosophy is somewhat different. We try to stay in $K$-theory without using the transition to cohomology using the Chern character. We hope to convince the reader that some of the arguments are more transparent in $K$-theory.

In the first part we introduce three different notions of the $K$-class of a projective subvariety, and show by examples that they are different. We explain their connection with classical algebraic geometric invariants as the Hilbert function and polynomial, and the arithmetic genus. We discuss the equivariant version, too, study the transversality properties and how these properties connected to $K$-theoretic Thom polynomials.

The second part is an attempt to introduce the motivic version of the projective Thom polynomial. The cohomological projective Thom polynomial was introduced by Andr\'as N\'emethi, Rich\'ard Rim\'anyi and the author in \cite{fnr-forms} and used later in other projects. I hope that this motivic version will be just as useful in applications.

I am grateful for Andrzej Weber for patiently explaining me the intricacies of the motivic Chern class. I had inspiring conversations on the topic with Rich\'ard Rim\'anyi, Andr\'as N\'emethi, \'Akos Matszangosz and  Bal\'azs K\H om\H uves. A special case of Theorem \ref{g-proj2affine} was first proved by the latter. I  thank Anders Buch for explaining the role of the Cohen-Macaulay condition in pulling back the sheaf theory $K$-class.

I was partially supported by NKFI 112703 and 112735 as well as ERC Advanced Grant LTDBud and enjoyed the hospitality of the R\'enyi Institute.

\section{What is the $K$-class of a subvariety?}
There are several candidates for the $K$-class of a subvariety of an ambient smooth variety $M$. We show that they are different and have different functorial properties.

\subsection{Algebraic $K$-theory and the sheaf $K$-class}
First we recall the basic constructions in algebraic $K$-theory following \cite[\S 15.1]{Fulton-intersection}:

For any scheme $ X$, $K^0 X$ denotes the Grothendieck group of vector bundles (locally free sheaves) on $X$. Each vector bundle $E$ determines an element, denoted by $[E]$, in $K^0 X$. $K^0 X$ is the free abelian group on the set of isomorphism
classes of vector bundles, modulo the relations
\[ [E] =[E'] + [E''],\]

whenever $E'$ is a subbundle of a vector bundle $E$, with quotient bundle
$E'' = E/E'$. The tensor product makes $K^0 X$ a ring. For any morphism $f:Y\to X$ there is an induced pull-back homomorphism
\[  f^*:K^0 X \to K^0 Y,  \]
taking $[E]$ to $[f^*E]$, where $f^*E$ is the pull-back bundle; this makes $K^0$ a contravariant functor from schemes to commutative rings.

The Grothendieck group of coherent sheaves on $X$, denoted by $K_0X$, is defined to be the free abelian group on the isomorphism classes  of
coherent sheaves on X, modulo the relations

\[ [\mathcal{F}] =[\mathcal{F}'] + [\mathcal{F}''],\]
for each exact sequence
\[  0 \rightarrow \mathcal{F}'\rightarrow\mathcal{F}\rightarrow\mathcal{F}''\rightarrow0\]
of coherent sheaves.

For any proper morphism $f:X\to Y$, there is a push-forward homomorphism
\[  f_*:K_0 X \to K_0 Y,  \]
which takes $[\mathcal{F}]$ to $\sum_{i\geq0}(-1)^i[R^if_*\mathcal{F}]$, where $R^if_*\mathcal{F}$ is Grothendieck's higher
direct image sheaf.

On any X there is a canonical \quot{duality} homomorphism:
\[ K^0X\to K_0X\]

which takes a vector bundle to its sheaf of sections. When $X$ is non-singular,
this duality map is an isomorphism.
 The reason for this is that a coherent
sheaf $\mathcal{F}$  on a non-singular X has a finite resolution by locally free sheaves, i.e.,
there is an exact sequence
\[ 0\rightarrow E_n\rightarrow E_{n-1}\rightarrow\dots\rightarrow E_1\rightarrow E_0\rightarrow \mathcal{F}\rightarrow0  \]
with $E_0,\dots,E_n$ locally free. The inverse homomorphism from $K_0$ to $K^0$ takes $[\mathcal{F}]$ to $\sum_{i=0}^{n}(-1)^i[E_i]$, for such a resolution.

In this paper we only study the case when $X$ is non-singular, so we identify $K_0$ with $K^0$ and denote the pushforward by $f_!$. We can define the \emph{sheaf $K$-class} of a subvariety $Y$ by $[\mathcal{O}_Y]\in K_0(X)$.

\subsection{Topological $K$-theory} Topological $K$-theory is a complex oriented cohomology theory, which has several consequences.  Any complex vector bundle $E\to X$ has an Euler class $e(E)\in K_{\text{top}}(X)$. (In this notation we incorporated the fact that $K_{\text{top}}$ is 2-periodic.) The Euler class of a line bundle $L$ is given by $e(L)=1-[L^*]$. Similarly to ordinary cohomology a complex submanifold $Y$  of the  complex manifold $M$ represent a class $[Y\subset M]\in K_{\text{top}}(M)$. Given a complex vector bundle $E$ with a section $\sigma:M\to E$ transversal to the zero section we have $e(E)=[\sigma^{-1}(0)\subset M]$. We have an obvious map from algebraic $K$-theory to topological $K$-theory, which is an isomorphism for $\P^n$, so in the remaining of the paper we identify these rings and also drop the upper and lower 0 indices.

\begin{theorem}\label{top=alg-push}
The forgetful map $K\to K_{\text{top}}$ respects pushforward,
\end{theorem}
 by \cite{hirzebruch1962riemann}, in particular for a complex submanifold $Y$  of the  complex manifold $M$ we have
\[ [Y\subset M]=[\mathcal O_Y].\]
The main goal of this section is to explore how to define this class for non smooth subvarieties.
\subsection{The $K$-theory of $\mathbb{P}^n$} $K(\P^n)=\Z[t]/((1-t)^{n+1})$, where $t=[\gamma]$, the class of the tautological line bundle. The corresponding sheaf is $\mathcal O(-1)$. The dual bundle is $L=\mathcal O(1)$. $L$ has sections transversal to the zero section so we get that
\[ e(L)=1-t=[\P^{n-1}\subset \P^n ].\]
We will also denote this class by $H$, the class of the hyperplane. Therefore we also have the description $K(\P^n)=\Z[H]/(H^{n+1})$.

\subsection{Hilbert polynomial} We show now that for the subvariety $X\subset\P^n$ the class $[\mathcal O_X]$ contains the same information as the Hilbert polynomial of $X$.

For $X\subset \P^n$ let $S=\C[x_0,\dots,x_n]$ denote the ring of polynomials, $I_X\lhd S$ the ideal of $X$ and $S(X):=S/I_X$ the homogeneous coordinate ring of $X$. The coordinate ring is a graded ring $S(X)=\bigoplus S^j(X)$ and we would like to encode the dimensions $h_j(X):=\dim S^j(X)$ (i.e. $j\mapsto h_j(X)$ is the Hilbert function of $X$). Notice that the embedding of $X$ is encoded in the grading of $S(X)$. It is well-known that there is a unique polynomial $p_X(x)$---the Hilbert polynomial of $X$---such that $h_j(X)=p_X(j)$ for $j\gg0$. For us it will be more convenient to use the Hilbert series
\[HS(X):=\sum_{j=0}^{\infty} h_j(X)t^j.\]
For example for $X=\P^n$ we have $h_j(X)=\binom{n+j}{n}$ which is clearly a polynomial of degree $n$ in $j$. The coefficients are certain Stirling numbers. On the other hand the Hilbert series has a particularly simple form:
\begin{equation}\label{binomial}
HS(\P^n)=\sum \binom{n+j}{n}t^j=\frac{1}{(1-t)^{n+1}}.
\end{equation}

The key property of the Hilbert series is (see e.g. in \cite{miller2004combinatorial})
\begin{theorem}\label{kpoly}
  There is a unique polynomial (the $\mathcal{K}$-polynomial) $\mathcal{K}_X(t)$, such that
 \[HS(X)=\frac{\mathcal{K}_X(t)}{(1-t)^{n+1}}.\]
\end{theorem}
Now we can state the proposed connection of the sheaf $K$-class with the Hilbert polynomial (see e.g. \cite[\S 21]{dugger2014geometric}):
\begin{proposition}\label{oandk}
$[\mathcal{O}_X]=\mathcal{K}_X(t)$ in $K(\P^n)=\Z[t]/((1-t)^{n+1})$.
\end{proposition}

Notice that adding a multiple of $(1-t)^{n+1}$ to $\mathcal{K}_X(t)$ changes only finitely many $h_j$'s, so the Hilbert polynomial doesn't change. $\mathcal{K}_X(t)$ encodes the Hilbert function and  $[\mathcal{O}_X]$ encodes the Hilbert polynomial. In this sense the natural generalization of the Hilbert polynomial for $X\subset M$ is the $K$-class $[\mathcal{O}_X]\in K(M)$.

\begin{example} It is not difficult to calculate (see \cite{harris2013algebraic}), that for three generic points $X\subset \P^2$ we have $h_X(j)=3$ for all $j>0$ (notice that $h_X(0)=1$ for all nonempty $X$!), and for three collinear points $Y\subset \P^2$ we have $\ h_Y(1)=2$ and $h_Y(j)=3$ for all $j>1$. Therefore
\[HS(X)=\left(3\sum t^j\right)-2=\frac{3}{1-t}-2=\frac{3(1-t)^2-2(1-t)^3}{(1-t)^3}\]
and
\[ HS(Y)=HS(X)-t=\frac{3(1-t)^2-(t+2)(1-t)^3}{(1-t)^3},\]
so their  Hilbert polynomial i.e. their sheaf $K$-class  agrees, but their $\mathcal{K}$-polynomial is different. We will see later that the $\mathcal{K}$-polynomial can be interpreted as the $\GL(1)$-equivariant sheaf $K$-class of the cone.
\end{example}
 Using standard resolution techniques (Koszul complex) one can show that
 \begin{corollary} \label{CI} Let $X=(f_1,\dots,f_k)\subset\P^n$ be a complete intersection. Then the sheaf-theoretic $K$-class
\[[\mathcal{O}_X]=\prod_{i=1}^{k}(1-t^{d_i}),\]
where $d_i$ is the degree of the generator $f_i$.
\end{corollary}

\begin{remark} Theorem \ref{kpoly} and Proposition \ref{oandk} implies that the assignment $X\mapsto p_X(m)$ induces a map  of $K(\P^n)$ to $\Q(m)/(m^{n+1})$, the truncated ring of the possible Hilbert polynomials, but this is only an additive homomorphism. According to \eqref{binomial} the base change is given by
\[H^n\mapsto 1,\ H^{n-1}\mapsto m+1,\ H^{n-2}\mapsto \frac12m^2+\frac32m+1,\ H^{n-3}\mapsto \frac16m^3+m^2+\frac{11}{6}m+1,\]
etc.
\end{remark}
The dimension and the degree of $X$ can be read off from the Hilbert polynomial. In fact some people use this fact to define the dimension and the degree. Translating this connection to the $H$-variable we get
\begin{corollary} \label{deg-dim} The sheaf $K$-class of $X\subset \P^n$ has the form $[\mathcal{O}_X]=\sum_{i=d}^{n}q_iH^i$, where $d$ is the codimension of $X$ and $q_d=\deg(X)$.
\end{corollary}

\begin{remark}
  If  the ideal of $X\subset \P^n$ is known, then the 'hilbert$\_$numerator' command of Sage (Singular) calculates $[\mathcal{O}_X]$. Also Maple has the 'HilbertSeries' which has to be multiplied by the denominator $(1-t)^{n+1}$. These calculations are feasible only for small examples.
\end{remark}

\subsection{The pushforward $K$-class} This is also only defined for closed subvarieties: Take a resolution $\varphi:\tilde X\to X\subset Y$, where $Y$ is smooth.
\[ [X]:=\varphi_*[\mathcal{O}_{\tilde X}].   \]
It is a non-trivial fact that this class is independent of the resolution. By Theorem \ref{top=alg-push} we have
\begin{proposition}
  \[ [X]=\varphi_!1,\]
for the $K$-theory pushforward.
\end{proposition}

For singular $X$ this is not easy to calculate. Even if we know a resolution, the $K$-group of $\tilde X$ can be complicated. However for $X=\bigcup_{i=1}^k X_i$ where $X_i$ are the irreducible components of $X$, we clearly have
 \begin{equation}\label{add}
 [X]=\sum _{i=1}^k [X_i],
 \end{equation}
 which helps if the components are smooth. Almost by definition we have
 \begin{theorem} If $X$ is irreducible with only rational singularities, then
 \[ [\mathcal{O}_X]=[X].\]
 \end{theorem}
 Indeed, in this case $R^if_*\mathcal{O}_{\tilde X}=0$ for $i>0$, and $R^0f_*\mathcal{O}_{\tilde X}=\mathcal{O}_X$ always holds for $X$ normal.

 \subsection{Motivic invariants and the motivic $K$-class}
The motivic $K$-class is defined as the constant term of the motivic Chern class:
\begin{definition}\cite{brasselet2010hirzebruch}  Suppose that $X\subset Y$, with $Y$ smooth, then the motivic $K$-class of $X$ is
  \[ \mC_0(X):=\mC(X)_{y=0},\]
  where the ambient manifold $Y$ is not denoted if it is clear from the context.
\end{definition}

A simple consequence of the definition is that $\mC_0(X)=[X]$ for smooth $X$.

Let us recall what motivic invariants are.
The surprising (and not so easy to prove) fact (Fulton) is that $\chi(W)=\chi(W\setminus U)+\chi(U)$ for any $U\subset W$ for (quasiprojective) varieties over $\C$. Over the reals this is not true: e.g. $W=\R$, $U=\{0\}$. We will call these type of invariants \emph{motivic}.

Main examples are the  Chern-Schwartz-MacPherson (CSM) and the motivic Chern class. There are several variations, but I prefer now the following setup: Let $h^*$ be a complex oriented cohomology theory (ordinary cohomology for CSM and $K$-theory for motivic Chern class). Then a motivic class for $h^*$ is a functor
\[m(U\subset M)\in h^*(M)\]
( or in $h^*(M)[y]$ for the motivic Chern class) for pairs of varieties (note that $U$ in not necessarily closed in $M$, it is a constructible subset) if it has the \emph{motivic property}: $m(W)=m(W\setminus U)+m(U)$ for any $U\subset W$, and a property I will call \emph{homology property}: Suppose that $f: M\to N$ is proper and $f$ is an isomorphism restricted to $U\subset M$, then
\begin{equation}\label{homology-property}\tag{$!$}
f_!m(U\subset M)=m(f(U)\subset N).
\end{equation}
This property means that $m(U\subset M)$ essentially depends only on $U$ and the dependence on the embedding is very simple. For example the fundamental cohomology class of $U$ has this property, and the reason for that is that there is a fundamental homology class as well. We restrict ourselves to $M$ smooth though it is not necessary. This property is also called covariant functoriality.

We claim that any motivic class in the above sense is determined by its value on closed submanifolds. Indeed, for any pair $U\subset M$ with $U,M$ smooth but $U$ not necessarily closed in $M$ there is a proper map $f:\tilde M\to M$ such that
\begin{enumerate}
  \item $f|_{f^{-1}(U)}: f^{-1}(U)\to U$ is an isomorphism,
  \item The \idez{divisor} $D:=\tilde M\setminus f^{-1}(U)$ is the union of closed submanifolds $D_i,\ i=1,\dots,s$ such that for all $I\subset\underline{s}=\{1,2,\dots,s\}$ the intersection $D_I:=\bigcap_{i\in I}D_i$ is a submanifold.
\end{enumerate}
By the weak factorization theorem such a map (called proper normal crossing extension) always exists. Then, by the motivic and the homology property we have:
\begin{equation} \label{pnc}
m(U\subset M)=\sum_{I\subset \underline{s}}\;(-1)^{|I|} f^{I}_!m(D_I)\,,
\end{equation}
where $f^{I}=f|_{D_I}$, and we use the convention that $m(M):=m(M\subset M)$.

For a smooth variety $M$  we define $\mc(M):=\lambda_y(T^*M)$, where for any complex vector bundle $\lambda_y(E)=\sum_{i=0}^{\rank E}[\Lambda^iE]y^i$. Notice that $\lambda_y(E)$ is a natural analogue of the total Chern class of cohomology theory. The existence of the motivic Chern class is proved in \cite{brasselet2010hirzebruch}. The definition of the equivariant version is again straightforward, see  in \cite{feher2018motivic}.

For $m=\mC_0$ we define $\mC_0(M)=1$ for smooth $M$. The existence of the motivic class $\mC_0$ follows from the existence of $\mc$, since $\mC_0(X)=\mc(X)_{y=0}$.
\begin{example}
\begin{table}[tbp]
\caption{$K$-classes of plane cubics }
\label{cubic-table}
\begin{tabular}{|p{4.8cm}|c|p{2.7cm}|c|c|c|}
\hline
 \makebox[0pt]{$\phantom{\Big\|}$}{\bf description of  $X$} &{\bf shape}& {\bf representant} & $[\mathcal{O}_X]$ & $[X]$ & $\mC_0(X)$ \\[2mm]
\hline
 nodal cubic & \raisebox{-1ex}{\node} & $\phantom{\Big\|}\!\!x^3+y^3+xyz$              & $3H-3H^2\ \ $ & $3H-2H^2\ \ $& $3H-3H^2\ $\\[2mm]
\hline
 cuspidal cubic & \raisebox{-1ex}{\cusp} & $\phantom{\Big\|}\!\!x^3+y^2z$              & $3H-3H^2\ \ $ & $3H-2H^2\ \ $& $3H-2H^2\ $\\[2mm]
\hline
conic and intersecting line & \raisebox{-1ex}\omeg & $\phantom{\Big\|}\!\!x^3+xyz$   & $3H-3H^2\ \ $ & $3H-H^2\ \ $& $3H-3H^2\ $\\[2mm]
\hline
conic and tangent line & \raisebox{-1ex}\tangent & $\phantom{\Big\|}\!\!x^2y+y^2z$   & $3H-3H^2\ \ $ & $3H-H^2\ \ $& $3H-2H^2\ $\\
\hline
 three nonconcurrent lines & \raisebox{-1ex}\tri & $\phantom{\Big\|}\!\!xyz$         & $3H-3H^2\ \ $ & $3H\ \ $& $3H-3H^2\ $\\
\hline
 three concurrent lines & \raisebox{-1ex}\concurrent & $\phantom{\Big\|}\!\!x^2y+xy^2$ & $3H-3H^2\ \ $ & $3H\ \ $& $3H-2H^2\ $\\
\hline
\end{tabular}
\end{table}

Let $X$ be a projective cubic plane curve. Then its sheaf-class is $1-t^3=3H-3H^2$. For the smooth cubics all three classes agree. For the others see Table \ref{cubic-table}. The calculations are straightforward for the reducible ones. For the nodal and cuspical ones we use the fact that they have a rational resolution $\varphi:\P^1\to \P^2$ of degree 3. First we need to calculate $\varphi_!1$. Since $\varphi_!1$ depends only on the degree, we have
\[\varphi_!1=i_!f_!1,\]
where $i:\P^1\to \P^2$ is the linear inclusion and $f:\P^1\to \P^1$ has degree 3. Since $i$ is injective, it is easy to see that $i_!1=H$ and $i_!H=H^2$. Less obvious is to calculate $f_!1$:
 \begin{lemma} \label{deg-d-push} Let $f_d:\P^1\to \P^1$ has degree $d$. Then
     \[ f_{d!}1=d-(d-1)H.\]
 \end{lemma}
  \begin{proof} Let $v_d:\P^1\to \P^d$ be the degree $d$ Veronese embedding. The image $C_d$ is smooth, so $v_{d!}1=[C_d\subset \P^d]$. Using the Hilbert polynomial or Corollary \ref{rncurve} we can see, that $[C_d\subset \P^d]=H^{d-1}\big(d-(d-1)H\big)$. Let $i_d: \P^1\to \P^d$ be the linear embedding. Then $i_df_d$ is homotopic to $v_d$, so $v_{d!}1=i_{d!}f_{d!}1$. On the other hand $i_{d!}H^j=H^{j+d-1}$, implying our result.
  \end{proof}

 Therefore $\varphi_!1=i_!f_!1=i_!(3-2H)=3H-2H^2=[\raisebox{-.3ex}\node]=[\raisebox{-.4ex}\cusp]$. Using the definition \eqref{pnc} of the motivic class (notice that these resolutions are evidently normal crossing) we get that $\mC_0(\raisebox{-.4ex}\cusp)=\varphi_!1$ since $\varphi$ injective in this case, and $\mC_0(\raisebox{-.3ex}\node)=\varphi_!1-H^2$ since $\varphi$ has a double point in this case.
\end{example}

\begin{theorem}\cite{brasselet2010hirzebruch} If $X$ has only Du Bois singularities then $[\mathcal{O}_X]=\mC_0(X)$. \end{theorem}
Also the definition of Du Bois singularity can be found in \cite{brasselet2010hirzebruch}. Additional information can be found in \cite{Weber=EquivariantHirzebruch}.
Important cases of Du Bois singularities are rational singularities, transversal union of smooth varieties (these are not rational) and cone hypersurfaces in $\C^n$ of degree $d \leq n$. We can see from our calculations that $\raisebox{-.3ex}\cusp$, $\raisebox{-.3ex}\tangent$ and $\raisebox{-.3ex}\concurrent$ are not Du Bois.

\begin{example} \label{2lines} Let $X=E^{n-k}\cup E^{n-l}\subset \P^n$, where $E^j$ is a $j$-dimensional projective subspace. We assume that $E^{n-k}$ and $ E^{n-l}$ are in general position and $k,l>0$. Then $X$ has only Du Bois singularities, therefore $[\mathcal{O}_X]=\mC_0(X)$. The latter can be easily calculated by the motivic property:
\[ \mC_0(X)=(1-t)^k+(1-t)^l-(1-t)^{k+l},\]
where the last term is 0 in the $K$-group if $k+l>n$, i.e. $E^{n-k}\cap E^{n-l}=\emptyset.$ On the other hand
\[  [E^{n-k}\cup E^{n-l}]=[E^{n-k}]+[E^{n-l}]=(1-t)^k+(1-t)^l.   \]
\bigskip

The calculation of $[\mathcal{O}_X]$ is usually done using the Hilbert syzygy theorem by calculating a resolution. For non obvious examples this is difficult, even with computers. Let us demonstrate this on $X=E^1\cup E^1\subset \P^3$. Then $I_X=(x,y)\cdot(w,z)=(xw,xz,yw,yz)$. The four relations are
\[\begin{array}{lcccc}
   g_1: & z & -w & 0 & 0 \\
   g_2: & 0 & 0 & z & -w \\
   g_3: & -y & 0 & x & 0 \\
    g_4: &0 & -y & 0 & x
  \end{array}\]
Finally we have a relation among the relations: $yg_1-xg_2+zg_3-wg_4$. This gives us
\[[\mathcal{O}_X]=1-4t^2+4t^3-t^4,\]
which is indeed congruent to $2(1-t)^2$ modulo $(1-t)^4$.


\end{example}
\subsection{The Todd genus} \label{sec:todd} We have already seen that understanding pushforward is essential in our calculations. Compared to cohomology, $K$-theory has a new feature. Let $\coll_X:X\to *$ denote the collapse map of $X$ for $X$ projective and smooth. Then the \emph{Todd genus}
\[\Td(X):={\coll_X}_!1=\int\limits_X1=\chi(X,\mathcal{O}_X)\in \Z\]
is a non-trivial invariant. (In this paper the integral sign will always denote the $K$-theory pushforward to the point.) This is a genus in the sense that it defines a ring homomorphism from the complex cobordism ring to the ring of integers.

As being a genus suggests it is enough to calculate $\td(\P^n)$ to be able to calculate more involved examples. It is a key result in $K$-theory, in particular the proof of Theorem \ref{top=alg-push} uses it. It is usually proved using the topological Grothendieck-Riemann-Roch theorem.
\begin{theorem}\cite{hirzebruch1962riemann} For any $n\in\N$ the Todd genus of the projective space $\P^n$ is 1:
  \[\td(\P^n)=1.\]
\end{theorem}

Then by basic properties of pushforward immediately yields:
\begin{corollary} \label{td-from-k} Let $X\subset \P^n$ be smooth with $K$-theory fundamental class
\[  [X]=\sum_{i=0}^{n}q_iH^i=q(H).\]
Then $\td(X)=\sum_{i=0}^{n}q_i=q(1)$.
\end{corollary}
\begin{proof} Notice first that  $\td(X)=\int\limits_{\P^n}[X]$, then we can apply the integral formula
\begin{equation}\label{integral}
  \int\limits_{\P^n}\sum_{i=0}^{n}q_iH^i=q(1).
\end{equation}
\end{proof}
Recalling the connection of $[X]=[\mathcal{O}_X]$ with the Hilbert polynomial $p_X$ we have that $q(1)=p_X(0)$. Recall that the \emph{arithmetic genus} is defined as
\[p_a(X):=(-1)^{\dim X}(p_X(0)-1),\]
so we see that it is essentially the same as the Todd genus for smooth $X$. In other words the Todd genus of $X$ is equal to its \emph{holomorphic Euler characteristics} $\chi(X,\mathcal{O}_X)$.
\subsection{Genus of smooth hypersurfaces}
For a smooth degree $d$ hypersurface $X_d\subset \P^n$ we have $[X]\equiv 1-t^d$, so a simple binomial identity implies that
\begin{corollary} \label{td-of-Zd}The arithmetic genus of the smooth degree $d$ hypersurface $X_d\subset \P^n$ is $\binom{d-1}{n}$.
\end{corollary}
Notice that by definition $\binom{d-1}{n}=0$ for $d\leq n$. This is the first sign that hypersurfaces of degree higher that $n$ behaves very differently then the low degree ones.
\begin{corollary} \label{rncurve} For the degree $d$ rational normal curve $X_d\subset \P^d$ we have $[X_d\subset \P^d]=dH^{d-1}-(d-1)H^d$.
\end{corollary}
Indeed, by Corollary \ref{deg-dim}. we have $[X_d\subset \P^d]=dH^{d-1}+q_dH^d$. But $X_d\iso\P^1$ so $1=\td(X_d)=d+q_d$.\qed
\subsection{The $\chi_y$ genus}
A straightforward extension of the Todd genus is the $\chi_y$-genus of Hirzebruch:
\[\chi_y(X):=\int\limits_X\lambda_y(T^*X),\]
for $X$ projective and smooth. In general if $X \subset M$ for $M$ projective and smooth we can define
\begin{equation}
  \label{tdy-def}\chi_y(X):=\int\limits_M\mC(X\subset M),
\end{equation}
which is independent of the embedding of $X$ (the reader is encouraged to check this), providing a motivic extension of the  $\chi_y$ genus. Clearly, substituting $y=0$ into the $\chi_y$ genus we obtain the Todd genus if $X$ is smooth, and the holomorphic Euler characteristics is general.
\begin{example} \label{tdypn}  It is instructive to calculate $\chi_y(\P^n)$. Since we have the short exact sequence of vector bundles
\[ 0\to \C\to L^{n+1}\to T\P^n\to 0,\]
and $\lambda_y$ is multiplicative, we have
\[\mC(\P^n)=\frac{(1+yt)^{n+1}}{1+y}\equiv\frac{(1+yt)^{n+1}-(-y)^{n+1}(1-t)^{n+1}}{1+y}=\sum_{i=0}^{n}(1+yt)^{i}(y(t-1))^{n-i}.\]
Applying the integral formula \eqref{integral} by substituting $t=0$ we arrive at
\[ \chi_y(\P^n)=1-y+y^2-\dots\pm y^n,\]
which was already calculated by Hirzebruch.
\end{example}

\begin{remark}\label{mcpn}
  It is interesting to write $\mC(\P^n)$ in the form $\sum_{i=0}^{n}q_iH^i$ for $q_i\in \Z[y]$ (See e.g. in \cite[\S 22]{dugger2014geometric}):
  \[ \mC(\P^n)=\sum_{i=0}^{n}\binom{n+1}{i}(-y)^i(1+y)^{n-i}H^i.\]
\end{remark}
For the next example we recall the \quot{divisor trick}, the multiplicativity of $\lambda_y$ implies the following:
\begin{corollary} \label{divisor} Suppose that $Y$ is the zero locus of a section of a vector bundle $E\to M$, which is transversal to the zero-section. Then the motivic Chern  class is
 \[\mc(Y\subset M)=e(E)\lambda_y(-E)\mc(M). \]
\end{corollary}
\begin{example} A smooth degree $d$ hypersurface $Z_d\subset \P^n$ is the zero locus of a section of the line bundle $(\gamma^*)^d$. Using the divisor trick we get
\[\mc(Z_d)=\frac{(1+yt)^{n+1}}{1+y}\cdot \frac{1-t^d}{1+yt^d}.\]
A closed formula for the $\chi_y$ genus gets complicated, so we give the answer for small $n$ only:
\[ n=2:\ \  \chi_y(Z_d)=\left(\binom{d-1}{2}+1\right)(y-1), \]
\[n=3:\ \ \chi_y(Z_d)=\left(\binom{d-1}{3}+1\right)(y-1)^2+\left(2\binom{d-1}{3}-4\binom{d}{3}+2\right)y.\]
\medskip

Substituting $y=0$ we can see that it is consistent with \ref{td-of-Zd}.
\end{example}

\section{Cones}The simplest singularities in some sense are the conical ones. We study two closely related cases.
\subsection{The projective case}
\begin{proposition}\label{xhat}
Let $X\subset \P^n$ be smooth and denote its cone in $\P^{n+1}$ by $\hat X$. Then for the unique polynomials $q_i\in \Z[y]$ and $\hat q_i\in \Z[y]$ such that $\mc(X\subset\P^n)=\sum_{i=0}^{n}q_iH^i$ and $\mc(\hat X\subset \P^{n+1})=\sum_{i=0}^{n+1}\hat q_iH^i$ we have
\[\hat q_i=(1+y)q_i-yq_{i-1}\ \text{ for }i=0,\dots,n\  \  \text{ and }\ \hat q_{n+1}=1-yq_n-(1+y)\chi_y(X),\]
where $q_{-1}=0$ and as we mentioned before $\chi_y(X)=\sum_{i=0}^{n}q_i$.
\end{proposition}
\begin{proof}
First notice that $j:\P^n\to\P^{n+1}$ is transversal to $\hat X$ and intersect in $X$. We assumed that $X$ is smooth, so we have the pullback formula:
\begin{equation}\label{pullback4cone}
  \frac{\mc(X)}{\mc(\P^n)}=j^*\left(\frac{\mc(\hat X)}{\mc(\P^{n+1})}\right).
\end{equation}
\begin{definition}The motivic Chern class has its Segre version just as the Chern-Schwartz-MacPherson class:
\[\ms(X\subset M):=\frac{\mc(X\subset M)}{\mc(M)}.\]
\end{definition}
Multiplicativity of $\lambda_y$ implies that the motivic Segre class behaves nicely with respect to transversal maps:
\begin{proposition} \label{segre-trans} Let $f:A\to M$ be a proper (in the topological sense) map of smooth varieties such that $f$ is transversal to $X\subset M$ in the sense that it is transversal to the smooth part of $X$ and does not intersect the singular part of $X$. Then
\[ \ms(f^{-1}(X)\subset A)=f^*\ms(X\subset M).  \]
\end{proposition}

Using that $j^*(H)=H$ and $\frac{j^*\mc(\P^{n+1})}{\mc(\P^n)}=1+yt$ we get that
\begin{equation}\label{firstncoef}
\hat q_i=(1+y)q_i-yq_{i-1}\ \text{ for }i=0,\dots,n.
\end{equation}

To calculate $\hat q_{n+1}$ we consider the blowup at the vertex $0$ of the cone. Restricting to the preimage of the cone  we get a  normal crossing resolution $\varphi:Y\to \hat X\subset \P^{n+1}$, where $ Y$ is a fiber bundle over $X$ with fiber $\P^1$. From the definition \eqref{pnc} we have
\begin{equation}\label{cone-with-reso}
 \mc(\hat X\setminus 0)= \varphi_!\lambda_y(Y)-\chi_y(X)H^{n+1}.
\end{equation}
The first term is difficult to calculate directly so we push forward \eqref{cone-with-reso} to a point:
\begin{equation}\label{toddyofcone}
  \chi_y(\hat X\setminus 0)=\chi_y(Y)-\chi_y(X).
\end{equation}
Now we use that $Y\to X$ is the projective bundle of a vector bundle, so
 \[ \chi_y(Y)=\chi_y(\P^1)\chi_y(X)=(1-y)\chi_y(X), \]
 implying
 \[\chi_y(\hat X)=1-y\chi_y(X). \]
(This product property of $\chi_y$ was already known by Hirzebruch, see e. g. \cite{hirzebruch-topological}.) On the other hand $\chi_y(\hat X)=\sum_{i=0}^{n+1}\hat q_i$ and $\chi_y( X)=\sum_{i=0}^{n} q_i$, so we can express $\hat q_{n+1}$ using \eqref{firstncoef}.
 \end{proof}

 \begin{remark} Proposition \ref{xhat}. immediately generalises to $X$ being a constructible set, if we use the motivic extension of the $\chi_y$ genus, i.e. we take \eqref{tdy-def} as the definition of the $\chi_y$ genus for any constructible set.
 \end{remark}

Substituting $y=0$ we get the following:
\begin{corollary}\label{m0xhat} Express $\mC_0(\hat X\subset \P^{n+1})$ and $\mC_0(X\subset\P^n)$ in their reduced form, i.e as polynomials of degree at most $n+1$ and $n$, respectively, in the variable $H$. Then
  \[\mC_0(\hat X\subset \P^{n+1})=\mC_0(X\subset\P^n)+(1-\td(X))H^{n+1}.\]
\end{corollary}

\begin{remark} The other two $K$-classes of a cone can also be calculated.
 For the pushforward $K$-class we can use the same resolution as above:
 \[ [\hat X]=\varphi_!1,\]
so Proposition \ref{xhat} implies that
\[ [\hat X]=\mC_0(\hat X)+(\td(X)-1) H^{n+1}=[X].\]
So the three $K$-classes of the cone differs only in the top coefficient, which is 0 for the pushforward class and $(1-\td(X))$ for the motivic $K$-class.

For the sheaf theoretic $K$-class of $\hat X$ we need more than the corresponding one  for $X$. It follows from the definition, that
$\mathcal{K}_{\hat X}(t)=\mathcal{K}_{ X}(t)$.
Write $\mathcal{K}_{ X}(t)=\sum p_iH^i$ as a polynomial (of degree possibly much higher than $n+1$) of $H=1-t$. Then the reduced form of the sheaf theoretic $K$-class of $X$ and $\hat X$ are  $\sum_{i=0}^{n} p_iH^i$, and $\sum_{i=0}^{n+1} p_iH^i$, respectively.
\end{remark}

\begin{remark} For CSM classes the calculation is simpler: for $\csm(X):=\sum_{i=0}^{n}q_iH^i\in H^*(\P^n)$ and $\csm(\hat X):=\sum_{i=0}^{n+1}\hat q_iH^i\in H^*(\P^{n+1})$ we get $\hat q_i=q_i+q_{i-1}$ for $i\leq n$ and $\hat q_{n+1}=\chi(\hat X)=q_n+1=\chi(X)+1$.
\end{remark}
\begin{example}
Corollary \ref{m0xhat} allows us to find irreducible examples of varieties for which all 3 $K$-classes are different. Let $X=Z_d\subset\P^2$ a smooth curve of degree $d$. Then $\mC_0(Z_d)\equiv1-t^d\equiv dH-\binom{d}{2}H^2$. Then
\[\mC_0(\hat Z_d)=dH-\binom{d}{2}H^2+\binom{d-1}{2}H^3.\]
On the other hand
\[ [\mathcal{O}_{\hat Z_d}]\equiv1-t^d\equiv dH-\binom{d}{2}H^2+\binom{d}{3}H^3,\]
and
\[ [\hat Z_d]=\mC_0(Z_d\subset\P^n)+H^{n+1}=dH-\binom{d}{2}H^2\]
by the previous remark.

Therefore all these 3 classes of $\hat Z_d$ are different if $d\geq4$. It implies that these hypersurfaces are not Du Bois. In fact it is known that $\hat Z_d\subset \P^{n+1}$ is Du Bois if and only if $d\leq n+1$, so this calculation detects all the non Du Bois cases among the $\hat Z_d$'s.
\end{example}

\begin{remark} Pushing forward the three $K$-classes to the point we get 3 different extensions of the Todd genus to singular varieties. $\int_{\P^n}[\mathcal{O}_X]$ is the holomorphic Euler characteristics. $\int_{\P^n}[X]=\td(\tilde X)$ is the Todd genus of the resolution (note that this is independent of the resolution!). We can call $\chi_{y=0}:=\int_{\P^n}\mC_0(X)$ the \emph{motivic Todd genus} of $X$. Our calculations show that for $X=\hat Z$, the projective cone of the smooth variety $Z$, $\int_{\P^n}[X]=\td(\tilde X)=\td(Z)$ and $\chi_{y=0}(X)=1$, so even these three Todd genus extensions are different for the projective cone of a smooth curve of degree $d$ if $d\geq4$. A similar example was discovered in \cite[ex 14.1]{Weber=EquivariantHirzebruch}.
\end{remark}


\section{Equivariant classes}
If an algebraic linear group $G$ acts on $M$, then we can define the Grothendieck group $K^G_0(M)$ of coherent $G$-sheaves. Also we can define  the Grothendieck group $K_G^0(M)$ of $G$-vector bundles. For smooth $M$ they are isomorphic. For $M=\C^n$ and $\P^n$ the forgetful map to $K^G_{\text top}(M)$---the Grothendieck group  of topological $G$-vector bundles---is also isomorphism. The ring $K(B_GM)$ is much bigger, we will not use it. The definition of the equivariant motivic Chern classes and the various equivariant $K$-classes are straightforward, one can repeat the same definitions equivariantly. For the existence of the equivariant motivic Chern class and more details about equivariant $K$-theory see e.g. \cite{feher2018motivic}. Notice that compared to cohomology the transition to the equivariant theory is much smoother, we do not need classifying spaces and approximation of the classifying spaces by algebraic varieties.

There are several reasons to introduce equivariant theory in this context. One is that implicitly we are already using scalar equivariant objects: vector bundles admit a canonical scalar action therefore they admit equivariant Euler class, which in cohomology can be identified with the total Chern class and in $K$-theory with $\lambda_y$ of the vector bundle, which is the starting point of building the motivic Chern class.

The second reason is that the definition of the Hilbert function is based on a scalar action: the homogeneous coordinate ring of $X\subset \P^n$ is graded according to the natural scalar action on it. This implies the following reformulation:

Consider first the scalar $\Gamma:=\GL(1)$-action on $\C^{n+1}$. For any $X\subset \P^n$ the cone $CX\subset \C^{n+1}$ is $\Gamma$-invariant, so its structure sheaf $\mathcal{O}_{CX}$ is a coherent $\Gamma$-sheaf.
\begin{proposition} \label{hilbert-function}The Grothendieck group $K^{\Gamma}_0(\C^{n+1})$ is isomorphic to $\Z[t,t^{-1}]$ via the restriction map to the origin, and
\[ [\mathcal{O}_{CX}]_{\Gamma}=\mathcal{K}_X(t).\]

\end{proposition}
The first statement is proved in \cite[5.4.17]{chriss-ginzburg}. The statement on the $\mathcal{K}$-polynomial is folklore, see \cite[p. 172]{miller2004combinatorial}. It is essentially equivalent to \cite[6.6.8]{chriss-ginzburg}.
Consequently the Kirwan-type homomorphism  $K^{\Gamma}_0(\C^{n+1})\to K_0(\P^n)$ with $t\mapsto t$ maps   $\mathcal{K}_X(t)=[\mathcal{O}_{CX}]_{\Gamma}$ to $[\mathcal{O}_X]$, i.e. assigns the Hilbert polynomial to the Hilbert function.

\begin{remark}\label{rem:awkward}
Notice, that $t$ as an element in the representation ring of $\Gamma:=\GL(1)$ is the inverse of the standard representation. It looks awkward first but this is the choice which reflects that the hyperplane is the zero locus of a section of the \emph{ dual} of the tautological bundle. A more conceptual explanation will be given in \ref{sec:Kirwan}.
\end{remark}
The third reason to introduce the equivariant theory might be the most important: Equivariant motivic, $K$, etc classes on $G$-invariant subvarieties are \emph{universal classes for degeneracy loci}.  Let us explain this statement in more details. An introduction into the cohomological theory can be found e.g. in \cite[\S 2.]{thomseries}  and for the CSM case in \cite{csm} and \cite{ohmoto-sing-char}, so we concentrate on the $K$-theory cases here.

\subsection{Universal classes in $K$-theory}
Let $G$ be a connected linear algebraic group and suppose that $\pi_P:P\to M$ is a principal $G$-bundle over the smooth $M$ and $A$ is a smooth $G$-variety. Then we can define a map
 \[a:K_G(A)\to K(P\times_G A)  \]
by association: For any $G$-vector bundle $E$  over $A$ the associated bundle $P\times_G E$ is a vector bundle over $P\times_G A$.

\begin{proposition} Let $Y\subset A$ be $G$-invariant. Then
\[ \mS(P\times_G Y\subset P\times_G A)=a\big(\mS_G(Y\subset A)\big).\]
\end{proposition}
The proof can be found in \cite[Pr 8.7]{guang}. Substituting $y=0$ we get the corresponding statement for the motivic $K$-class:

\begin{proposition} Let $Y\subset A$ be $G$-invariant. Then
\[ \mC_0(P\times_G Y\subset P\times_G A)=a\big(\mC_0^G(Y\subset A)\big).\]
\end{proposition}

Similarly one can show the analogous statement for the push forward $K$-class:
\begin{proposition} Let $Y\subset A$ be a $G$-invariant subvariety. Then
\[ [P\times_G Y\subset P\times_G A]=a[Y\subset A]_G.\]
\end{proposition}
For the proof you need to check that for a $G$-equivariant resolution $\varphi:\tilde Y\to A$ of $Y$ the induced resolution $\hat \varphi:P\times_G \tilde Y\to P\times_G A$ of $P\times_G Y$ has the property $\hat \varphi_!1=a \varphi_!1$.

For the sheaf $K$-class we have
\begin{proposition} Let $Y\subset A$ be a $G$-invariant subvariety. Then
\[ [\mathcal{O}_{P\times_G Y}]=a[\mathcal{O}_Y]_G.\]
\end{proposition}

In most applications $A$ is a vector space and a section $\sigma:M\to P\times_G A$  sufficiently transversal to $P\times_G Y$ is given. For example

\begin{corollary} \label{locus}Suppose that $\sigma:M\to P\times_G A$ is a section motivically transversal to $P\times_G Y$. Then
\[ \mS(Y(\sigma)\subset M)=a\big(\mS_G(Y\subset A)\big),\]
where $Y(\sigma)=\sigma^{-1}(P\times_G Y)$ is the $Y$-locus of the section $\sigma$.

If $A$ is a vector space then we identify the $K$-theory of $M$  with the $K$-theory of $P\times_G A$ via $\sigma^*$.
\end{corollary}

For the proof and the definition of  motivically transversal see \cite[\S 8]{guang}. The corollary implies the analogous statement for $\mC_0$.

For the push forward $K$-class a weaker transversality condition is sufficient: we only need that the pullback of the resolution $\varphi:\tilde Y\to A$ by $\sigma$ is a resolution of $\sigma^{-1}(Y)$.

Recently Rimanyi and Szenes studied the $K$-theoretical Thom polynomial of the singularity $A_2$ in \cite{rimanyi2018residues}. They choose the push forward $K$-class which means that for a reasonably wide class of maps their $K$-theoretical Thom polynomial calculates the  push forward $K$-class of the $A_2$-locus. It would be interesting to study the motivic version of their $K$-theoretical Thom polynomial.

For the sheaf $K$-class the conditions are more complicated. $Y$ has to be Cohen-Macaulay of pure dimension (in many applications like \cite{rimanyi2018residues} this is not satisfied). If $\sigma^{-1}(Y)$ is also of pure dimension and its codimension agrees with the codimension of $Y$ then $[\mathcal{O}_{Y(\sigma)}]=a[\mathcal{O}_{Y}]_G$, where $Y(\sigma)$ is the pull back scheme. To ensure that $Y(\sigma)$ is reduced we need further transversality conditions.

\subsection{Equivariant classes of cones in cohomology: the projective Thom polynomial}
Earlier we explained the connection between the motivic Chern class of $X\subset\P^n$ and of its  projective cone. Just as interesting is the case of the affine cone, however we are forced to use equivariant setting, otherwise there is not enough information in the class of the affine cone.

Suppose that a complex torus $\T$ of rank $k$ acts on $\C^{n+1}$ linearly, i.e. a homomorphism $\rho:\T\to\GL(n+1)$ is given. We assume that the action \emph{contains the scalars}: there is a non zero integer $q$ and a homomorphism $\varphi:\GL(1)\to \T$ such that $\rho\varphi(z)=z^qI$ for all $z\in\GL(1)$. Suppose that $X\subset\P^n$ is $\T$-invariant. Then $CX\subset \C^{n+1}$ is also $\T$-invariant and we can compare their various classes. The first such connection was found about the equivariant cohomology class in \cite{fnr-forms}, what we recall now.

After reparamerization of $\T$ we can assume that $\varphi(z)=\diag(z^{w_1},\dots,z^{w_k})$, where the integers $w_1,\dots,w_k$ are the weights of $\varphi$. Then we have the following:
\begin{proposition} \label{cx2x-coho} The $\T$-equivariant cohomology class
  \[ [X\subset\P^n]\in H^*_{\T}(\P^n)=\Z[a_1,\dots,a_k][x]/(\prod_{i=1}^{n+1}(b_i-x)), \]
where $x=c_1^{\T}(\gamma)$ is the equivariant first Chern class of the tautological bundle with the induced $\T$-action and $b_i$ are the weights of the $\T$-action on $\C^{n+1}$,  can be expressed from the $\T$-equivariant cohomology class
\[[CX\subset \C^{n+1}]\in H^*_{\T}(\C^{n+1})=\Z[a_1,\dots,a_k]\]
by the substitution
\[ [X]=\sub([CX],a_i\mapsto a_i-\frac{w_i}{q}x).\]
\end{proposition}

This formula has several useful applications, in particular it helps to calculate the degree of certain subvarieties (see e.g. \cite{fnr-degree}). There is a counterpart which is quite obvious in cohomology:
\begin{proposition} \label{x2cx-coho} The $\T$-equivariant cohomology class
 \[[CX\subset \C^{n+1}]_{\T}\in H^*_{\T}(\C^{n+1})=\Z[a_1,\dots,a_k]\]
  can be expressed from the $\T$-equivariant cohomology class
\[ [X\subset\P^n]_{\T}\in H^*_{\T}(\P^n)=\Z[a_1,\dots,a_k][x]/(\prod_{i=1}^{n+1}(b_i-x)) \]
by the substitution
\[ [CX\subset \C^{n+1}]_{\T}=\sub([X]_{\T},x\mapsto 0),\]
where the the substitution is done into the reduced form of $[X]_{\T}$, the unique polynomial of $x$ degree at most $n$ representing $[X]_{\T}$ in $\Z[a_1,\dots,a_k][x]$.
\end{proposition}
\subsection{Projective Thom polynomial for the motivic Chern class} There is an analogous result for the motivic Chern class. First we need to understand $K_{\T}(\P^n)$.

\subsubsection{The Kirwan map in $K$-theory}
\label{sec:Kirwan} We have a Kirwan-type surjective map $\kappa: K_{\Gamma\times\T}(\C^{n+1})\to K_{\T}(\P^n)$. More generally  let $V$ be a $\Gamma\times\T$-vector space for a connected algebraic group $\Gamma$ and assume that $P\subset V$ is an open $\Gamma\times\T$-invariant subset such that $\pi:P\to P/\Gamma$ is a principal $\Gamma$-bundle over the smooth $M:=P/\Gamma$.

Strictly speaking $\Gamma$ acts on $V$ from the left, and on $P$ on the right, so we need to define
\begin{equation}\label{l2r}
  pg:=g^{-1}p
\end{equation}
  for all $g\in \Gamma$ and $p\in P\subset V$.

Notice first that we have a restriction map $r:K_{\Gamma\times\T}(V)\to K_{\Gamma\times\T}(0)$. Then for a $\Gamma\times\T$-representation $W$ we can apply the association map $W\to P\times_\Gamma W$ to induce a map $a:K_{\Gamma\times\T}(0)\to K_{\T}(M)$, and we can define $\kappa:=ar$.

Specializing to $V=\C^{n+1}$, $\Gamma=\GL(1)$ acting as scalar multiplication, $P=\C^{n+1}\setminus 0$ we obtain $\kappa: K_{\Gamma\times\T}(\C^{n+1})\to K_{\T}(\P^n)$. Notice that the switch between left and right action explained above implies that $\kappa(t)=[\gamma]_{\T}$ for $t$ denoting the inverse of the standard representation of $\Gamma=\GL(1)$, explaining the convention in Remark \ref{rem:awkward}.

The $\T$-bundle $\Hom(\gamma,\C^{n+1})$ over $\P^n$ has a  nowhere zero $\T$-equivariant section (the inclusion of $\gamma$ into the trivial bundle $\P^n\times \C^{n+1}$), therefore
\[ e_{\T}\big(\Hom(\gamma,\C^{n+1})\big)=\prod_{i=1}^{n+1}(1-t/\beta_i)=0.\]
It can be checked that this is the only relation, therefore
\[ K_{\T}(\P^n)\iso \Z[\alpha_1,\alpha_1^{-1},\dots,\alpha_k,\alpha_k^{-1}][t,t^{-1}]/(\prod_{i=1}^{n+1}(1-t/\beta_i)).\]

The relation can be rewritten as $\prod_{i=1}^{n+1}(t-\beta_i)=0$ which implies that any element  $\omega\in  K_{\T}(\P^n)$ can be written uniquely as a polynomial of degree at most $n$ in $t$ with coefficients in $ K_{\T}$ (i.e. $\P^n$ is equivariantly formal in $K$-theory for linear $\T$-actions). We call this polynomial the \emph{reduced form} of $\omega$.

\subsubsection{The affine to projective formula}
The analogue of Proposition \ref{cx2x-coho} for  motivic classes is similar, and the proof is essentially the same:
\begin{theorem} \label{cx2xsegre-k} The $\T$-equivariant motivic Segre class
  \[ \ms_{\T}(X\subset\P^n)\in K_{\T}(\P^n)[y]=\Z[\alpha_1,\alpha_1^{-1},\dots,\alpha_k,\alpha_k^{-1}][t,t^{-1}]/(\prod_{i=1}^{n+1}(1-t/\beta_i))[y], \]
where $t=[\gamma]_{\T}$ is the class of the tautological bundle with the induced $\T$-action and $\beta_i$ are the characters of the $\T$-action on $\C^{n+1}$,  can be expressed from the $\T$-equivariant motivic Segre class
\[   \ms_{\T}(C_0X\subset \C^{n+1})\in K_{\T}(\C^{n+1})[y]=\Z[\alpha_1,\alpha_1^{-1},\dots,\alpha_k,\alpha_k^{-1}][y]  \]
by the substitution
\[ \ms_{\T}(X)=\sub(\ms_{\T}(C_0X),\alpha_i\mapsto \alpha_i\cdot t^{-\frac{w_i}{q}}),\]
where $C_0X=CX\setminus 0$.
\end{theorem}

It is natural to use the motivic Segre class because it has the transversal pull back property. We need to use $C_0X$ instead of $CX$, because $C_0X$ is the preimage of $X$ under the quotient map $\C^{n+1}\setminus 0\to\P^n$.

Strictly speaking the motivic Segre class lives in the completion of $K_{\T}(\P^n)[y]$ because of the division with $\mC(\P^n)$, but we will not denote this completion. At the end we are mainly interested in the motivic Chern class where the completion is not needed.
\begin{proof}
 First notice that by a simple change of variables we have
\begin{proposition} If the $\T$-action contains the scalars as above then
  \[ \ms_{\Gamma\times\T}(Z)=\sub(\ms_{\T}(Z),\alpha_i\mapsto \alpha_i\cdot t^{-\frac{w_i}{q}})\]
  for any $\T$-invariant (therefore $\Gamma\times\T$-invariant) constructible subset $Z\subset \C^{n+1}$.
\end{proposition}
On the other hand we have
\begin{proposition}\label{git}
\[ \ms_{\T}(X)=\kappa\ms_{\Gamma\times\T}(C_0X),\]
\end{proposition}
as a special case of \cite[Thm 8.12]{guang}.
\end{proof}

We can translate the result to motivic Chern class easily:
\begin{theorem} \label{cx2x-k} The $\T$-equivariant motivic Chern class of $X\subset\P^n$ can be calculated via the substitution
\[ \mC_{\T}(X)=\frac{1}{1+y}\sub(\mC_{\T}(C_0X),\alpha_i\mapsto \alpha_i\cdot t^{-\frac{w_i}{q}}).\]
\end{theorem}
This formula probably can be used to calculate the Hilbert polynomial of certain subvarieties of the projective space.

\begin{corollary}\label{T=1} Applying Proposition \ref{git} to the trivial torus we get
\[\ms(X)=\kappa\ms_\Gamma(C_0X),\]
implying that
 \[   \mC(X)=\frac{1}{1+y}\kappa\mC_{\Gamma}(C_0X). \]
\end{corollary}

\subsection{The projective to affine formula} Interestingly, the calculation of $\mC_{\T}(C_0X)$ from $\mC_{\T}(X)$ is more involved than the obvious formula of Proposition \ref{x2cx-coho} for the corresponding cohomology classes. The projective cone formula \ref{xhat} already indicates the subtleties ahead of us. Finding such formula is important since some motivic Chern class calculations are simpler in the projective case than in the affine case, as unpublished works of B. K\H om\H uves show. In this section we give an \quot{inverse} to Theorem \ref{cx2x-k}:
\begin{theorem} \label{gxscalar-proj2affine} Suppose that the  torus $\T$ acts linearly on $\C^{n+1}$ and $X\subset \P^n$. Then
  \[\mC_{\Gamma\times \T}(C_0X)=(1+y)\big(\mC_{\T}(X)-\chi_y(X)[0]_{\Gamma\times \T}\big), \]
  where $[0]_{\Gamma\times \T}$ is the $\Gamma\times \T$-equivariant $K$-class of the origin, and $\mC_{\T}(X)$ is written in the reduced form in the variable $t$, the $\T$-equivariant class of the tautological bundle $\gamma$.
\end{theorem}
Notice that the $\chi_y$ genus has no $\T$-equivariant version. This is called the \emph{rigidity} of the $\chi_y$ genus, see e.g. \cite[pr. 7.2]{Weber=EquivariantHirzebruch}.

Before proving the theorem let us have a look at the case when $\T$ is the trivial torus.
\begin{corollary} \label{scalar-proj2affine} Suppose that $X\subset\P^n$ is a constructible subset. Then the affine cone minus the origin $C_0X\subset \C^{n+1}$ is invariant for the scalar action of $\Gamma=\GL(1)$, and
  \[\mC_{\Gamma}(C_0X)=(1+y)\big(\mC(X)-\chi_y(X)[0]\big), \]
  where $[0]=(1-t)^{n+1}$ is the $\Gamma$-equivariant $K$-class of the origin, and $\mC(X)$ is written in the reduced form in the variable $t$, the $K$-theory class of the tautological bundle $\gamma$.
\end{corollary}
An other way to express Corollaries \ref{git} and \ref{scalar-proj2affine} together is that written in the variable $H=1-t$ the coefficients of $(1+y)\mC(X)$ (in the reduced form) and $\mC_{\Gamma}(C_0X)$ are the same, except $\mC_{\Gamma}(C_0X)$ has also an $n+1$'st coefficient to assure that  the sum of the coefficients is zero.

Comparing with Proposition \ref{hilbert-function} we can see an important difference between the sheaf $K$-class and the motivic $K$-class: The scalar-equivariant motivic $K$-class of the cone of $X$ contains no additional information than the motivic $K$-class of $X$, on the other hand the scalar-equivariant sheaf $K$-class of the cone of $X$ determines the Hilbert function not just the Hilbert polynomial of $X$.

\begin{example} Let us study the case $X=\P^k\subset\P^n$. Then by Remark \ref{mcpn} and the simple fact that $i_!(H^j)=H^{j+n-k}$ for the inclusion $i:X\to\P^n$ we have
  \[ \mC(\P^k\subset\P^n)=\sum_{i=0}^{k}\binom{k+1}{i}(-y)^i(1+y)^{k-i}H^{n-k+i}.\]
  For the cone we have the product formula (see \cite[\S 2.7]{guang})
  \[  \mC_\Gamma(\C^{k+1}\subset\C^{n+1})=(1-t)^{n-k}(1+yt)^{k+1},  \]
  so removing the origin we get
  \[  \mC_\Gamma(C_0X\subset\C^{n+1})=(1-t)^{n-k}(1+yt)^{k+1}-(1-t)^{n+1}. \]
  We know from Example \ref{tdypn} that $\chi_y(X)=1-y+y^2-\cdots\pm y^k=\frac{1-(-y)^{k+1}}{1+y}$ so Corollary \ref{scalar-proj2affine} gives the identity
  \[ \left((1+y)\sum_{i=0}^{k}\binom{k+1}{i}(-y)^i(1+y)^{k-i}H^{n-k+i}\right)-(1-(-y)^{k+1})H^{n+1}=(1-t)^{n-k}(1+yt)^{k+1}-(1-t)^{n+1}, \]
  where $H=1-t$, which of course can be checked directly.
\end{example}

 The idea of the proof of Theorem \ref{gxscalar-proj2affine} is that we first prove it for the special case of $X$ being a projective space, and show how this result implies the result for general $X$.

\begin{example} \label{ex:proj} Suppose that the torus $\T$ acts on $\C^{n+1}$ with characters $\beta_1,\dots,\beta_{n+1}$. Let  $X=\P^k\subset \P^n$ be invariant for the induced $\T$-action on $\P^n$. Without loss of generality we can assume that $CX$ is spanned by the first $k+1$ eigenvectors. Then
 \[  \mC_{\Gamma\times\T}(\C^{k+1}\setminus0\subset \C^{n+1})  = M-R,\]
 where
 \[M:=\prod_{i=1}^{k+1}\left(1+\frac {yt}\beta_i\right)\prod_{i=k+2}^{n+1}\left(1-\frac t\beta_i\right),\]
 and
 \[R:=[0]_{\Gamma\times \T}=\prod_{i=1}^{n+1}\left(1-\frac t\beta_i\right),\]
 where $R$ is also the relation in $K_{\T}(\P^n)$ after identifying $t$ with the $\T$-equivariant class of $\gamma$. On the other hand we have
 \[ (1+y)\mC_{\Gamma\times\T}(\P^k\subset \P^n)\equiv M,\]
 but this is not the reduced form yet, the coefficient of $t^{n+1}$ is not zero. Comparing $M$ and $R$ we can  see that the reduced form is
  \[ (1+y)\mC_{\T}(\P^k\subset \P^n)= M-(-y)^{k+1}R,\]
so the right hand side of Theorem \ref{gxscalar-proj2affine} becomes
 \[ (1+y)(\mC_{\T}(\P^k\subset \P^n)-\chi_y(\P^k)R)= M-(-y)^{k+1}R-(1-(-y)^{k+1})R,\]
 since $(1+y) \chi_y(\P^k)=1-(-y)^{k+1}$. Consequently we see that Theorem \ref{gxscalar-proj2affine} holds for these examples.
\end{example}

The next step is to prove Theorem \ref{gxscalar-proj2affine} for $X$ being a $\T$-invariant smooth subvariety of $\P^n$. In this case the blowup of $\C^{n+1}$ at the origin provides a $\varphi:Y\to \C^{n+1}$ proper normal crossing extension for $C_0X$, where $Y$ is the total space of the restriction of the tautological bundle $\gamma\to\P^n$ to $X$. The resolution factors as
\[ \xymatrix{Y\ar[r]^(0.3)j&\P^n\times \C^{n+1}\ar[r]^(0.57)\pi &\C^{n+1}},\]
which implies that

\begin{equation}\label{eq-integral}
  \mC_{\Gamma\times\T}(C_0X)=\int_{\P^n} \mC_{\T}(X)\lambda_y(\gamma^*)e(\C^{n+1}/\gamma)\ -\ e(\C^{n+1})\int_{\P^n} \mC_{\T}(X),
\end{equation}
where the $\lambda_y$ class and the Euler classes are $\Gamma\times\T$-equivariant.

It is quite difficult to use \eqref{eq-integral} for calculations. Luckily we do not need it. We only need to notice that \eqref{eq-integral} implies that the left hand side can be calculated from the reduced form of $\mC_{\T}(X)$ providing a $K_{\T}[y]$-module homomorphism. This implies that it is enough to check \ref{gxscalar-proj2affine} for a basis of the space of polynomials of degree at most $n$ in the variable $t$ and coefficients in $K_{\T}[y]$. We claim that the cases of Example \ref{ex:proj} will give such a basis. Indeed, substituting $\beta_i=1$ and $y=0$ we get $\mC_0(\P^k\subset\P^n)=(1-t)^{n-k}$.

The last step is to extend the result to all $\T$-invariant constructible subsets of $\P^n$. For that we just have to notice that all 3 components of the formula are motivic and we finished the proof of Theorem \ref{gxscalar-proj2affine}.\qed

Forgetting the scalar action we still get a nontrivial statement:
\begin{theorem} \label{g-proj2affine} Suppose that the  torus $\T$ acts linearly on $\C^{n+1}$. Then
  \[\mC_{ \T}(C_0X)=(1+y)\big(\mC_{\T}(X)|_{t=1}-\chi_y(X)[0]_{ \T}\big), \]
  where $[0]_{ \T}$ is the $\T$-equivariant $K$-class of the origin, and $\mC_{\T}(X)$ is written in the reduced form in the variable $t$.
\end{theorem}

\bibliography{k-classes}
\bibliographystyle{alpha}

\end{document}